\documentclass[letterpaper, 10 pt, conference]{ieeeconf}  

\IEEEoverridecommandlockouts                              

\usepackage{graphicx}
\usepackage{enumerate}
\usepackage{float}
\usepackage[font=footnotesize]{caption}

\usepackage{pgfplots}
\usepackage{amssymb}
\usepackage{amsmath}

\usepackage{tikz}
\usepackage{comment}

\usepackage{mathtools, cuted}

\usetikzlibrary{automata,arrows,positioning,calc}

\usepackage{todonotes}

\newtheorem{theorem}{Theorem}

\newtheorem{example}{Example}
\newtheorem{definition}{Definition}
\newtheorem{proposition}{Proposition}

\title{\LARGE \bf
	The Price of Anarchy for Transportation Networks \\ with Mixed Autonomy
}

\author{Daniel A. Lazar$^{1}$, Samuel Coogan$^{2}$ and Ramtin Pedarsani$^{3}$
	\thanks{$^{1}$Daniel Lazar is with the Department of Electrical and Computer Engineering, 
		University of California, Santa Barbara
		{\tt\small dlazar@ece.ucsb.edu}}%
	\thanks{$^{2}$Samuel Coogan is with the School of Electrical and Computer Engineering and the School of Civil and Environmental Engineering, 
		Georgia Institute of Technology
		{\tt\small sam.coogan@gatech.edu}}%
	\thanks{$^{3}$Ramtin Pedarsani is with the Department of Electrical and Computer Engineering, 
		University of California, Santa Barbara
		{\tt\small ramtin@ece.ucsb.edu}}%
}

\begin{document}

	\maketitle
	\thispagestyle{empty}
	\pagestyle{empty}

	\begin{abstract}
		
		We study routing behavior in transportation networks with mixed autonomy, that is, networks in which a fraction of the vehicles on each road are equipped with autonomous capabilities such as adaptive cruise control that enable reduced headways and increased road capacity. Motivated by capacity models developed for such roads with mixed autonomy, we consider transportation networks in which the delay on each road or link is an affine function of two quantities: the number of vehicles with autonomous capabilities on the link and the number of regular vehicles on the link. We particularly study the price of anarchy for such networks, that is, the ratio of the total delay experienced by selfish routing to the socially optimal routing policy. Unlike the case when all vehicles are of the same type, for which the price of anarchy is known to be bounded, we first show that the price of anarchy can be arbitrarily large for such mixed autonomous networks. Next, we define a notion of asymmetry equal to the maximum possible travel time improvement due to the presence of autonomous vehicles. We show that when the degree of asymmetry of all links in the network is bounded by a factor less than 4, the price of anarchy is bounded. We also bound the bicriteria, which is a bound on the cost of selfishly routing traffic compared to the cost of optimally routing additional traffic. These bounds depend on the degree of asymmetry and recover classical bounds on the price of anarchy and bicriteria in the case when no asymmetry exists. Further, we show with examples that these bound are tight in particular cases.
		
		
	\end{abstract}

	\section{INTRODUCTION}
	
	Automobiles are increasingly equipped with autonomous and semi-autonomous technologies such as adaptive cruise control and automated lane-keeping. These technologies are often marketed to consumers as safety or convenience features, but it is apparent that increasing numbers of these \emph{smart} vehicles will have dramatic impact on network-level mobility factors such as traffic congestion and travel times \cite{dot:2015zr}. A primary mechanism whereby such autonomous capabilities can improve mobility is by enabling \emph{platooning} of groups of smart vehicles along the roadway. A platoon consists of two or more vehicles which are able to automatically maintain short headways between them using, \emph{e.g.}, adaptive cruise control (ACC), which allows a vehicle to use radar or LIDAR to automatically maintain a specified distance to the preceding vehicle, or cooperative adaptive cruise control (CACC) which augments ACC with vehicle-to-vehicle communication.
	
	When all vehicles in the system are smart, platooning has the potential to increase network capacity by as much as three-fold \cite{Lioris:2015lq}. Platooning can help to smooth traffic flow and avoid \emph{shockwaves} of slowing vehicles \cite{Shladover:1978lo,Darbha:1999dw,Yi:2006hb, Pueboobpaphan:2010qe, Orosz:2016hb}, and at signalized intersections, platoons can synchronously accelerate at green lights \cite{Lioris:2015lq, Askari:2016fy}. However, in a \emph{mixed} autonomy setting---where only a fraction of the vehicles are smart and the remainder are \emph{regular}, human-driven vehicles---the benefits of platooning are less clear. On freeways, simulation results suggest that high penetration rates of smart vehicles are required to realize significant improvement in traffic flow \cite{Vander-Werf:2002fh, Van-Arem:2006ai, Jiang:2007rr, Yuan:2009th, Kesting:2010wd, Arnaout:2014ul}.
	
	In our prior research, we developed an analytic model for the capacity of roads at signalized intersections with mixed autonomous traffic \cite{lazar2017routing}. There, we consider a queue of vehicles at an intersection and suppose that smart vehicles platoon opportunistically, that is, if a smart vehicle queues behind another smart vehicle, they maintain a short headway along the road. We also proposed and studied a second model in which each smart vehicle maintains a short headway to the preceding vehicle, whether it is also smart or not. Such a scenario may be increasingly possible as adaptive cruise control with passive sensing continues to improve. Our capacity models describe the maximum possible flow rate of vehicles through an intersection as a nonlinear function of the \emph{level of autonomy} of the road, that is, the fraction of smart vehicles on the road. 
	
	In this work, we use the capacity models that we developed in \cite{lazar2017routing} in order to study routing behavior on road networks. We make the assumption that the additional travel time caused by congestion on a road is inversely related to capacity and proportional to the total number of vehicles on the road. Given a network of roads leading from origins to destinations, selfish vehicles will choose the route that minimizes total delay, achieving a Wardrop equilibrium \cite{wardrop1900some, Dafermos:1969qt}. It has long been known that a Wardrop equilibrium may be suboptimal in the sense that a social planner is able to prescribe routes that achieve a lower total delay for all vehicles in the network. The ratio of the socially optimal delay to the worst possible Wardrop equilibrium is called the \emph{price of anarchy} \cite{Koutsoupias:1999fs, papadimitriou2001algorithms}. For affine separable cost functions, when only one type of vehicle is present (\emph{i.e.}, no smart vehicles), it is known that the price of anarchy cannot exceed 4/3 \cite{roughgarden2002bad}.
	
	In a mixed autonomy setting, however, a social planner is able to route smart vehicles differently than regular vehicles to maximize capacity. In this paper, we first show that this increased flexibility leads, remarkably, to an unbounded price of anarchy. Next, we make the assumption that the possible travel time improvement due to the presence of autonomous vehicles is bounded by a factor $k<4$. We call this factor the \emph{degree of asymmetry} of the network. Under this assumption, we prove that the price of anarchy cannot exceed $\frac{4}{4-k}$, which recovers the classical bound when $k=1$, \emph{i.e.}, the case when smart vehicles do not enable any improvement in travel time. We additional show via examples that this bound is tight for $k=1$ and $k=2$.
	
	Next, we provide a bound on the cost of selfish routing relative to the cost of optimally routing additional traffic, called the bicriteria bound \cite{roughgarden2002bad}, \cite{roughgarden2002selfishthesis}. We prove that traffic at a Wardrop Equilibrium will not exceed the cost of optimally routing $1+\frac{k}{4}$ as much traffic of each type, where $k$ is the degree of asymmetry in the network. We demonstrate by example that the bicriteria bound is tight for $k=4$. Similar to the price of anarchy, when the asymmetry is unbounded we show that the bicriteria is unbounded as well. This runs counter to the case of single-type traffic where the bicriteria is bounded by $2$ for any continuous and nondecreasing cost function in which the delay on a road depends only on the traffic on that road \cite{roughgarden2002bad}.
	
	We also consider affine cost functions in which the delay on one road can affect that on another. We provide a bound on the price of anarchy and bicriteria for a class of these cost functions. 
	
	This paper is organized as follows. Section \ref{sct:prev_works} discusses prior work for routing games. Section \ref{sct:motivation_formulation} formally presents the problem formulation and demonstrates that the price of anarchy is unbounded for mixed autonomous traffic. In Section \ref{sec:bound}, we derive bounds on the price of anarchy given a bound on the degree of asymmetry of the network, and in Section \ref{sct:tightness}, we present several examples that demonstrate the tightness of our bound. Concluding remarks are provided in Section \ref{sec:conclusions}.
	
	\section{PREVIOUS WORKS}\label{sct:prev_works}
	
	In this section, we address related models in the literature and highlight the difference between these and our model. Due to the breadth of the field, we give a limited overview of the literature on the price of anarchy -- see \cite{correa2011wardrop} for a broader survey of literature related to Wardrop equilibria and \cite{roughgarden2005selfish} for a wider background on the price of anarchy in transportation networks. For definitions of the terms used in this section see Section \ref{sct:separability_monotonicity}.
	
	Roughgarden and Tardos \cite{roughgarden2002bad} bound price of anarchy and bicriteria for separable monotonic cost functions, and Roughgarden \cite{roughgarden2003topology} gives a more general method for determining price of anarchy in the separable case.
	
	Chau and Sim \cite{chau2003price} bound the price of anarchy for symmetric cost operators with convex social cost for both nonelastic and elastic demands.
	
	Perakis discusses nonseparable, asymmetric, nonlinear costs in \cite{perakis2007price}, though only for monotone latencies \emph{i.e.} satisfying the property 
	\begin{equation}\label{eq:monotone}
	\langle c(z) - c(v),z-v \rangle \ge 0 \; ,
	\end{equation}
	where $\langle \cdot , \cdot \rangle$ denotes the inner product of two vectors.
	
	Correa et. al \cite{correa2008geometric} present a unified framework for deriving price of anarchy and bicriteria for nonseparable monotone functions. 
	
	Unlike these previous works, we present a price of anarchy and bicriteria bound for a class of \emph{nonmonotone} and pairwise separable affine cost functions. We show that our bound simplifies to the classic bounds for affine monotone cost functions in \cite{roughgarden2002bad}, \cite{chau2003price}, and \cite{correa2008geometric} when there is no asymmetry in how the vehicle types affect congestion. 
	
	
	\section{MOTIVATION AND MATHEMATICAL FORMULATION}\label{sct:motivation_formulation}
	
	In this section we motivate our cost function for traffic in mixed autonomy. In Section \ref{sct:motivation}, we show that the price of anarchy and bicriteria are unbounded for congestion games with affine cost functions in mixed autonomy, described in Section \ref{sct:affine}. Prompted by our negative result, in Section \ref{sct:separability_monotonicity} we describe a pairwise separable cost function that is parameterized by the degree of asymmetry, as well as a more general class of nonseparable cost function.
	
	\subsection{Motivation}\label{sct:motivation}
	
	We provide a brief example of unbounded price of anarchy and bicriteria for congestion games under mixed autonomy. This example is similar in design to Pigou's example, as in \cite{roughgarden2002bad}, \cite{pigou1932economics}, and \cite{roughgarden2005selfish}. 
	
	
	\begin{example}\label{ex:unbounded}
		Consider the traffic network in Fig. \ref{fig:unbouded2road}, in which $\frac{1}{2\zeta}$ units of regular traffic and $\frac{1}{2}$ units of smart traffic need to travel from node $s$ to node $t$, where $\zeta \ge 1$. The cost on road $i$, or the delay a car experiences from traveling on that road, is denoted $c_i(x,y)$.
		
		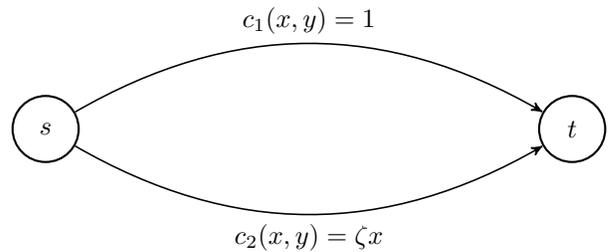
\begin{figure}
			\begin{center}
				\begin{tikzpicture}[->, >=stealth', auto, semithick, node distance=7cm]
				\tikzstyle{every state}=[fill=white,draw=black,thick,text=black,scale=1]
				\node[state]    (0)               {$s$};
				\node[state]    (1)[right of=0]   {$t$};
				\path
				(0) edge[bend left]		node{$c_1(x,y) = 1$}     (1)
				(0) edge[bend right]	node[below]{$c_2(x,y) = \zeta x$}     (1);
				\end{tikzpicture}
			\end{center} 
			\caption{Example of a road network with unbounded price of anarchy and bicriteria.}
			\label{fig:unbouded2road}
		\end{figure}	
		
		The routing at Wardrop Equilibrium has all traffic on the bottom road, for a price of $C^{EQ} = (\frac{1}{2\zeta}\zeta)(\frac{1}{2\zeta} + \frac{1}{2}) = \frac{1}{4 \zeta} + \frac{1}{4}$. The optimal routing has the regular traffic on the top and smart traffic on the bottom with a cost of $C^{OPT} = \frac{1}{2\zeta} \cdot 1 + \frac{1}{2} \cdot 0 = \frac{1}{2\zeta}$. This results in a price of anarchy of $(\zeta+1)/2$.
		
		For the bicriteria, consider a situation in which we have a mass of $\frac{a}{2\zeta}$ units of smart cars and $\frac{a}{2}$ units of regular cars to route. We want to find the $a$ that corresponds to, under optimal routing, a cost of $\frac{1}{4 \zeta} + \frac{1}{4}$. The optimal routing will have cost $\frac{a}{2\zeta}$, which equals $\frac{1}{4 \zeta} + \frac{1}{4}$ when $a=\frac{\zeta+1}{2}$.	
	\end{example}
	
	Here we see that both the price of anarchy and the bicriteria bound grow unboundedly with $\zeta$. Due to this result, we state the following proposition:
	\begin{proposition}
		The price of anarchy and bicriteria are unbounded for networks of mixed autonomy with pairwise separable affine functions.
	\end{proposition}
	
	Because of this negative result, in order to provide a bounded price of anarchy and bicriteria in mixed autonomy, we develop a class of cost functions parameterized by the maximum degree of asymmetry.
	
	\subsection{Affine Congestion Game Overview}\label{sct:affine}
	
	Consider a network of $n$ roads, with $m$ origin-destination pairs, each with an associated mass of regular vehicles of volume $\alpha_i$ and mass of smart vehicles of volume $\beta_i$. Since we are considering a nonatomic congestion game, each user controls an infinitesimally small portion of that mass. We denote $\mathcal{X}$ as the set of feasible strategies which result in the entirety of each mass being routed from its origin to its destination, without violating conservation of flow (see \cite{depalma1998optimization} for a more detailed explanation).
	
	The vector of all flows on the $n$ roads is denoted by 
	\begin{align*}
	z = \begin{bmatrix}
	x_1 & y_1 &	x_2 & y_2 & \ldots & x_n & y_n
	\end{bmatrix}^T \; ,
	\end{align*}
	where $x_i$ and $y_i$ represent the mass of regular and smart vehicles, respectively, on road $i$. In this paper, we consider affine cost functions, meaning the cost on the roads resulting from a routing $z \in \mathcal{X}$ can be written as
	\begin{align*}
	c(z) &= Az + b \; .
	\end{align*}
	where $A\in \mathbb{R}_{\ge 0}^{2n}$ and \begin{align*}
	b = \begin{bmatrix}
	b_1 & b_1 &	b_2 & b_2 & \ldots & b_n & b_n
	\end{bmatrix}^T \; .
	\end{align*} 
	
	This results in a social cost of
	\begin{align*}
	C(z) &= \langle c(z), z \rangle \\
	&= (Az + b)^Tz \; .
	\end{align*}
	
	The social cost at optimal routing is then
	\begin{align*}
	C^{OPT}=\inf_{z\in \mathcal{X}}C(z) \; .
	\end{align*}
	We see that the matrix $A$ is the Jacobian of the road cost operator, and the vector $b$ contains the constant terms. $A$ is not in general positive semidefinite, so this is not a convex optimization problem.
	
	\subsection{Separability and Monotonicity}\label{sct:separability_monotonicity}
	
	Having described the basic structure of the congestion game with affine costs, we describe the separability and monotonicity of our model. To do so, we define three notions of separability.
	\begin{definition}\label{df:separable}
		A cost function $c(z) = Az + b$ is \emph{separable} if $A$ is a diagonal matrix.
	\end{definition}
	\begin{definition}\label{df:pairwiseseparable}
		A cost function $c(z) = Az + b$ is \emph{pairwise separable} if $A$ is a blockwise diagonal matrix with 2x2 blocks.
	\end{definition}
	\begin{definition}\label{df:nonseparable}
		A cost function is \emph{nonseparable} if it is neither separable nor pairwise separable.
	\end{definition}
	
	It is clear that separable costs do not model mixed autonomy if regular and smart cars affect delay differently but experience it identically. The slightly more general class of pairwise separable costs does provide a useful model, which we motivate as follow:
	
	We use capacity model I in \cite{lazar2017routing}, which assumes that smart vehicles can platoon behind any vehicle. This results in a capacity (with units vehicles per hour) on road $i$ of
	\begin{align*}
	g_i(x_i,y_i) = \frac{m_i M_i (x_i + y_i)}{Mx_i + my_i}
	\end{align*}
	where $x_i$ and $y_i$ are the flows of regular and smart vehicles respectively, and $m_i$ and $M_i$ are the reciprocals of the time gap required in front of a regular vehicle and smart vehicle, respectively, on road $i$. We propose a cost function on road $i$ (with unit hours per vehicle) of the form 
	\begin{align*}
	c_i(x_i,y_i) &= b_i + r_i\frac{x_i+y_i}{g_i(x_i,y_i)} \\
	&= b_i + \frac{r_i}{m_i}x_i + \frac{r_i}{M_i}y_i \; ,
	\end{align*}
	in which the delay is affine with respect to $x_i$ and $y_i$. Here $b_i$ represents the time it takes to traverse a road in free-flow traffic and $r_i$ determines how congestion scales as road utilization increases with respect to capacity.
	
	In \cite{lazar2017routing}, \cite{milanes2014cooperative}, and \cite{lioris2017platoons}, we see that vehicles not in a platoon require approximately 2.5 times more headway than vehicles in a platoon. Motivated by this observation, we introduce the parameter $k_i \triangleq \frac{M_i}{m_i}$ that represents this factor on road $i$. Letting $a_i \triangleq \frac{r_i}{M_i}$, the cost on each road can be written as
	\begin{align}\label{eq:pairwiseSepCost}
	c_i(x_i,y_i) &= b_i + k_i a_i x_i + a_i y_i \; .
	\end{align}
	
	This leads us to a cost function of the following form:	
	\begin{align*}
	c(z) &= Az + b \\ 
	&= \begin{bmatrix}
	A_1 & 0 & \ldots & 0  \\
	0 & A_2 & \ldots  & 0  \\
	\vdots & \vdots  & \ddots & \vdots  \\
	0 & 0 & \ldots & A_n \\	
	\end{bmatrix}
	z + b
	\end{align*}
	where $A$ is a block-diagonal matrix with blocks
	\begin{align*}
	A_i &= \begin{bmatrix}
	k_i a_i & a_i \\
	k_i a_i & a_i \\
	\end{bmatrix} \; .
	\end{align*} 	
	
	The parameter $k_i$ allows us to represent the degree of asymmetry between the effect of regular and smart traffic on congestion on a road. We allow this factor to differ between roads, but generally expect it to be in the range $k_i \in [1,4]$.	
	
	We find it useful to parameterize a class of cost functions by its maximum degree of asymmetry, as follows:
	\begin{definition}\label{df:Ck}
		Let $\mathcal{C}_k$ denote the class of cost functions for which $\max(k_i, \frac{1}{k_i}) \le k$ $\forall i$ for some constant $k$. We call $k$ the \emph{maximum degree of asymmetry} of this class of cost functions.
	\end{definition}
	
	In the more general model explored in Section \ref{sct:bound_nonseparable}, the delay on one road may depend on the flows on other roads. For example, if one road is fully congested, the roads feeding it will have additional delay. If this is the case, then \eqref{eq:pairwiseSepCost} does not hold and the matrix $A$ is not of block-diagonal form. In Section \ref{sct:bound_nonseparable} we provide a bound for this model under certain circumstances.	
	
	Throughout this paper, we deal with cost functions that satisfy element-wise monotonicity, defined as follows:
	\begin{definition}\label{df:elementwise_monotonicity}
		A class of cost functions $\mathcal{C}$ is \emph{elementwise monotone} if for all cost functions $c(v)$ drawn from $\mathcal{C}$, $\frac{\partial c_j}{\partial v_i}(v) \ge 0$ $\forall i,j$.
	\end{definition}
	
	In other words, a cost function is element-wise monotonic if increasing any flow of vehicles will not decrease the delay on any road. This will be the case for a class of cost functions of the form $c(z) = Az + b$ for which $A$ has only nonnegative entries. Note that this is different from the more general notion of monotonicity in a function described in Section \ref{sct:prev_works}\footnote{Our case is not in general monotone: consider a single road with $c(z) = \begin{bmatrix}
		3 & 1 \\
		3 & 1 \\
		\end{bmatrix}z $, with $x = \begin{bmatrix}
		1 & 0
		\end{bmatrix}^T$ and $y=\begin{bmatrix}
		0 & 2
		\end{bmatrix}^T$. This results in $\langle c(x)-c(y),x-y\rangle = -1$.}. 		
	
	\section{BOUNDING THE PRICE OF ANARCHY}
	\label{sec:bound}
	In this section we present bounds for the price of anarchy and bicriteria. We proceed along the lines proposed in \cite{correa2008geometric}, reviewing that work in Section \ref{sct:bound_prelims} and highlighting the differences that arise for a nonmonotone cost function. We then derive our bounds for nonmonotone pairwise separable costs in Section \ref{sct:our_bound}, and give a bound for nonseparable costs in Section \ref{sct:bound_nonseparable}.
	
	\subsection{Preliminaries}\label{sct:bound_prelims}
	
	Smith \cite{smith1979existence} shows that any flow $z^{EQ}$ at Wardrop equilibrium -- in which all users sharing an origin and destination take paths of equal cost and no unused paths have a smaller cost -- will satisfy the variational inequality for all feasible flows $x$:
	\begin{align}\label{eq:VI}
	\langle c(z^{EQ}), z^{EQ}-z \rangle \le 0 \; .
	\end{align}
	
	A simple proof of this is provided in \cite{depalma1998optimization}.
	
	Correa et. al. \cite{correa2008geometric} use this result to develop a general tool for finding price of anarchy and bicriteria. To that end, they introduce the following parameters:
	\begin{align*}
	\beta(c,v) := \max_{z\in \mathbb{R}_{\ge 0}^{2n}} \frac{\langle c(v) - c(z), z\rangle}{\langle c(v),v \rangle}
	\end{align*}
	(where 0/0=0 by definition), and 
	\begin{align*}
	\beta(\mathcal{C}):= \sup_{c\in\mathcal{C},v\in \mathcal{X}} \beta(c,v) \; ,
	\end{align*}
	where $\mathcal{C}$ represents a class of cost functions and $\mathcal{X}$ is the set of feasible routings.
	
	In the following theorem, we adapt Correa \emph{et. al}'s Theorem 4.2 \cite{correa2008geometric} to when the cost function is not monotone. In the nonmonotone case $\beta(\mathcal{C})$ can be greater than 1, leading to an unbounded price of anarchy. For completeness, we overview the proof of Theorem 4.2 in \cite{correa2008geometric}.
	
	\begin{theorem}\label{thm:correa}
		Let $z^{EQ}$ be an equilibrium of a nonatomic congestion game with cost functions drawn from a class $\mathcal{C}$ of nonseparable nonmonotone but elementwise monotone cost functions.	
		\begin{enumerate}[(a)]
			\item If $z^{OPT}$ is a social optimum for this game, and $\beta(\mathcal{C})\le 1$, then $C(z^{EQ})\le (1-\beta(\mathcal{C}))^{-1}C(z^{OPT})$.
			\item If $w^{OPT}$ is a social optimum for the same game with $1+\beta(\mathcal{C})$ times as many player of each type, then $C(z^{EQ})\le C(w^{OPT})$.
		\end{enumerate}
	\end{theorem}
	\begin{proof}
		To prove part (a),
		\begin{align}
		\langle c(z^{EQ}),z \rangle &= \langle c(z), z\rangle + \langle c(z^{EQ})-c(z), z\rangle \nonumber \\
		&\le \langle c(z), z \rangle + \beta(c,z^{EQ})\langle c(z^{EQ}),z^{EQ} \rangle \nonumber \\
		&\le C(z) + \beta(\mathcal{C})C(z^{EQ}) \label{eq:lemma41}
		\end{align}	
		and $C(z^{EQ}) \le \langle c(z^{EQ}),z \rangle$. Completing the proof requires that $\beta(\mathcal{C}) \le 1$.
		
		To prove part (b), element-wise monotonicity implies the feasibility of $(1+\beta(\mathcal{C}))^{-1}w^{OPT}$, and using \eqref{eq:VI},
		\begin{align}
		\langle c(z^{EQ}),z^{EQ} \rangle \le \langle c(z^{EQ}),(1+\beta(\mathcal{C}))^{-1}w^{OPT} \rangle \; . \label{eq:thm1ln0}
		\end{align}
		Then,
		\begin{align}
		C(z^{EQ}) &= (1+\beta(\mathcal{C}))\langle c(z^{EQ}),z^{EQ} \rangle \nonumber \\
		& \quad - \beta(\mathcal{C})\langle c(z^{EQ}), z^{EQ} \rangle \label{eq:thm1ln1} \\
		&\le (1+\beta(\mathcal{C}))\langle c(z^{EQ}),(1+\beta(\mathcal{C}))^{-1}w^{OPT} \rangle \nonumber \\
		& \quad - \beta(\mathcal{C})\langle c(z^{EQ}), z^{EQ} \rangle \label{eq:thm1ln2}\\
		&\le C(w^{OPT}) \; , \label{eq:thm1ln3}
		\end{align}
		where \eqref{eq:thm1ln2} uses \eqref{eq:thm1ln0} and \eqref{eq:thm1ln3} uses \eqref{eq:lemma41}.
	\end{proof}
	
	\subsection{Pairwise Separable Costs}\label{sct:our_bound}
	
	We now present a bound for the price of anarchy and bicriteria for the pairwise separable affine cost function when $k$, the maximum degree of asymmetry of the cost function, is bounded. In particular, when $k < 4$, the price of anarchy is bounded, and the bicriteria is bounded for any $k$. This is formalized as follows:
	
	\begin{theorem}\label{thm:pairwise_separable}
		Let $z^{EQ}$ be an equilibrium of a nonatomic congestion game with cost functions drawn from a class $\mathcal{C}_k$ of affine pairwise separable nonmonotone, elementwise monotone cost functions where $k$ parameterizes the maximum degree of asymmetry in the cost functions.
		\begin{enumerate}[(a)]
			\item If $z^{OPT}$ is a social optimum for this game, and $k < 4$, then $C(z^{EQ})\le \frac{4}{4-k}C(z^{OPT})$.
			\item If $w^{OPT}$ is a social optimum for the same game with $1+\frac{k}{4}$ times as many player of each type, then $C(z^{EQ})\le C(w^{OPT})$.
		\end{enumerate}	
	\end{theorem}
	\begin{proof}
		To prove this, we will show $\beta(\mathcal{C}_k) \le \frac{k}{4}$ and then use Theorem \ref{thm:correa}. For ease of notation, let $z^{EQ} \triangleq \begin{bmatrix}
		x_1^* & y_1^* &	x_2^* & y_2^* & \ldots & x_n^* & y_n^*
		\end{bmatrix}$. 
		
		Without loss of generality, and with a slight abuse of notation, we order the roads such that for $1 \le i \le \ell$, $c_i(x_i,y_i)= k_i a_i x_i + a_i y_i$ and for roads $\ell < i \le n$, $c_i(x_i,y_i)= a_i x_i + k_i a_i y_i$, where $k_i \ge 1$. Then,
		\begin{align}
		\beta&(c,z^*) = \max_{z^* \in \mathbb{R}^{2n}_{\ge 0}} \frac{\langle c(z^*) - c(z), z \rangle}{\langle c(z^*), z^* \rangle} \nonumber \\ 
		&\le \frac{\max_{z\in \mathbb{R}^{2n}_{\ge 0}} \langle A(z^*-z),z \rangle}{\langle Az^*, z^* \rangle} \nonumber \\
		&= \frac{\sum_{i=1}^{\ell}a_i \max_{x_i,y_i\ge 0}(k_i (x_i^*-x_i)+(y_i^*-y_i))(x_i+y_i)}{\langle Az^*, z^* \rangle} \nonumber \\
		&\; + \frac{\sum_{i=\ell+1}^{n}a_i \max_{x_i,y_i\ge 0}((x_i^*-x_i)+k_i (y_i^*-y_i))(x_i+y_i)}{\langle Az^*, z^* \rangle}. \label{eq:thm2l0}
		\end{align}
		
		We will bound the first term in \eqref{eq:thm2l0}, but the following can be done for the second term as well. Consider the maximization in the numerator of the first term of \eqref{eq:thm2l0}:
		\begin{align*}
		&\max_{x_i,y_i\ge 0} \delta(x_i,y_i)  \quad \quad \text{where}\\
		&\delta(x_i,y_i) =(k_i (x_i^*-x_i)+(y_i^*-y_i))(x_i+y_i) \; .
		\end{align*}
		
		The function $\delta(x_i,y_i)$, is not concave for $k_i \neq 1$. Because of this we split it into two cases and take the maximum of the two.
		
		CASE 1:
		
		First we can consider maximizing with respect to $x_i$ in terms of $y_i$. 
		
		We see that $\frac{\partial^2 \delta}{\partial x_i^2}(x_i,y_i) = -2k$, so the function is concave with respect to $x_i$. Then, $\frac{\partial \delta}{\partial x_i}(x_i,y_i) = k_i x^*_i + y^*_i-2k_ix_i-(k+1)y_i$, giving us a minimum when $\tilde{x}_i(y_i) = \frac{k_i x^*_i +y^*_i -(k_i+1)y_i}{2k}$.
		
		Plugging this in, we have 
		\begin{align*}
		\max_{y_i \ge 0} \delta(\tilde{x}_i(y_i), y_i) = \max_{0 \le y_i \le \frac{k_i x^*_i + y^*_i}{k+1}}  \frac{(k_i x^*_i + y^*_i + (k-1)y_i)^2}{4k_i} \; .
		\end{align*}
		
		The expression above is clearly increasing in $y_i$. However, the condition that $x_i \ge 0$ implies that $y_i \le \frac{k_i x^*_i + y^*_i}{k+1}$. Using this value for $y_i$, we have
		\begin{align*}
		\max_{y_i \ge 0} \delta(\tilde{x}_i(y_i), y_i) = \frac{k (k x^*_i + y^*_i)^2}{(k+1)^2} \; .
		\end{align*}
		
		CASE 2:
		
		Next we consider maximizing with respect to $y_i$ in terms of $x_i$. We check that $\frac{\partial^2 \delta}{\partial x_i^2}(x_i,y_i) = -2$, then find $\frac{\partial \delta}{\partial x_i}(x_i,y_i) = k_i x^*_i + y^*_i-(k_i+1)x_i-2y_i$, giving us a minimum at $\tilde{y}_i(x_i) = \frac{k_ix^*_i + y^*_i - (k_i+1)x_i}{2}$. Using this, we have
		\begin{align*}
		\max_{x_i\ge 0} \delta(x_i,\tilde{y}_i(x_i)) = \max_{0 \le x_i\le \frac{k_i x^*_i + y^*_i}{2(k+1)}} \frac{(k_i x^*_i + y^*_i - (k-1)x_i)^2}{4} \; .
		\end{align*}
		
		We know that $y_i \ge 0$, so $x_i \le \frac{k_i x^*_i + y^*_i}{2(k+1)} \le k_i x^*_i + y^*_i$. Therefore, for the valid range of $x_i$, the overall expression decreases with $x_i$. We can then maximize the above expression by setting $x_i$ to 0. This provides
		\begin{align*}
		\max_{x_i\ge 0} \delta(x_i,\tilde{y}_i(x_i)) = \frac{(k_i x^*_i + y^*_i)^2}{4} \; .
		\end{align*}
		
		We see that case 2 gives us a greater value than case 1 for $k \ge 1$. We can use a similar analysis for roads $\ell < i \le n$ to find
		\begin{align}
		& \beta(c,z^*) \nonumber \\
		&\le \frac{1}{4} \frac{\sum_{i=1}^{\ell}\rho_i(k_i x^*_i + y^*_i) + \sum_{i=\ell+1}^{n}\sigma_i(x^*_i + k_i y^*_i)}{\sum_{i=1}^{\ell}\rho_i(x_i^*+y_i^*)+\sum_{i=\ell+1}^{n}\sigma_i(x_i^*+y_i^*)} \nonumber \\
		&= \frac{k}{4} \frac{\sum_{i=1}^{\ell}\rho_i(k_i x^*_i + y^*_i) + \sum_{i=\ell+1}^{n}\sigma_i(x^*_i + k_i y^*_i)}{\sum_{i=1}^{\ell}\rho_i(kx_i^*+ky_i^*)+\sum_{i=\ell+1}^{n}\sigma_i(kx_i^*+ky_i^*)} \nonumber \\
		&\le \frac{k}{4} \label{eq:thm2l1}\; ,
		\end{align}
		where $\rho_i \triangleq a_i(k_i x^*_i + y^*_i)$ and $\sigma_i \triangleq a_i(x^*_i + k_i y^*_i)$. The fact that $\frac{\sum_{i=1}^{n}w_i}{\sum_{i=1}^{n}v_i}\le 1$ when $0 \le w_i \le v_i$ implies Equation \eqref{eq:thm2l1}, since $k \ge k_i \ge 1$ $\forall i$. We apply Theorem \ref{thm:correa} to find a price of anarchy bound of $\frac{4}{4-k}$ and bicriteria bound of $1+\frac{k}{4}$.
	\end{proof}
	
	\subsection{Nonseparable costs}\label{sct:bound_nonseparable}
	
	Having discussed pairwise separable costs (Definition \ref{df:pairwiseseparable}), where the delay on each road depends only on the vehicles on that road, we now consider nonseparable costs (Definition \ref{df:nonseparable}). As an example, consider a series of roads, each one feeding into the next; if one road is fully congested, this will increase the delay on the roads feeding it, resulting in cascading congestion. Another scenario of nonseparable costs is when intersecting streets affect the traffic on each other \cite{correa2011wardrop}, such as in a signalized intersection that senses traffic and adjusts its duty cycle accordingly. In that case, the volume of traffic on a road will affect the delay on the road perpendicular to it.
	
	
	To put this in more concrete terms, consider a road feeding into another narrower road. We model the congestion on the second road as comparatively affecting that on the first road by a factor of $\mu$. This results in a cost function of
	\begin{align*}
	c(z)=\begin{bmatrix}
	k_1 a_1 & a_1 & \mu k_2 a_2 & \mu a_2 \\
	k_1 a_1 & a_1 & \mu k_2 a_2 & \mu a_2 \\
	0 & 0 & k_2 a_2 & a_2 \\
	0 & 0 & k_2 a_2 & a_2 
	\end{bmatrix}z + b \; .\\
	\end{align*}	
	
	With this motivation, we consider the affine cost functions $c(x) = (Ax)^T + b$, where $A$ is no longer a 2x2 block-diagonal matrix. We consider the case that $A$ can be written as the sum of a block diagonal matrix $Q$ (with 2x2 blocks) with strictly positive block diagonal entries and a positive definite matrix $P$. \footnote{Note that if $P$ is a diagonal dominant mapping, \emph{i.e.} $P_{ii} > \frac{1}{2}\sum_{j\neq i}|P_{ij} + P_{ji}|$, then it is positive definite \cite{depalma1998optimization}. In this case, in order to also guarantee that the block diagonal components of $Q$ have strictly positive entries, we require
		\begin{align*}
		A_{ii} &> \frac{1}{2} \sum_{j \neq i, i-1}|A_{ij} + A_{ji}| \; \text{for } i \text{ even} \; , \\
		A_{ii} &> \frac{1}{2} \sum_{j \neq i, i+1}|A_{ij} + A_{ji}| \; \text{for } i \text{ odd} \; .
		\end{align*}
		
		This, however, is a sufficient but not necessary condition for the application of Theorem \ref{thm:nonseparable}.} 

	We describe the bounds we can establish under these conditions in the following theorem:
	\begin{theorem}\label{thm:nonseparable}
		Let $z^{EQ}$ be an equilibrium of a nonatomic congestion game with cost function $c(z) = (Az)^T + b$. Suppose $A$ can be split into $Q$, which is a (2x2) block diagonal matrix with strictly positive entries on the block diagonal, and $P$, which is positive definite, such that $A = Q + P$. Let $k$ be the maximum degree of asymmetry for the cost function defined by $Q$.
		\begin{enumerate}[(a)]
			\item If $z^{OPT}$ is a social optimum for this game, and if $k < 4$, then $C(z^{EQ})\le (\frac{4}{4-k}+\eta^2)C(z^{OPT})$, where $\eta^2 = \lambda_{max}(S^{-1/2}P^TS^{-1}PS^{-1/2})$ and $S = (P+P^T)/2$.
			\item If $w^{OPT}$ is a social optimum for the same game with $2+\frac{k}{4}$ times as many player of each type, then $C(z^{EQ})\le C(w^{OPT})$.
		\end{enumerate}		
	\end{theorem}
	\begin{proof}
		For part (a), we split the price of anarchy into two components, as
		\begin{align}
		\frac{C^{EQ}}{C^{OPT}} &= \sup_{z\in \mathcal{X}} \frac{(Az^{EQ}+b)^Tz^{EQ}}{(Az+b)^Tz}  \nonumber \\
		&= \sup_{z\in \mathcal{X}} \frac{((Q+P)z^{EQ}+b)^Tz^{EQ}}{((Q+P)z+b)^Tz} \nonumber \\
		&\le \sup_{z\in \mathcal{X}} \frac{(Qz^{EQ}+b)^Tz^{EQ}}{(Qz+b)^Tz} + \sup_{z\in \mathcal{X}} \frac{(Pz^{EQ})^Tz^{EQ}}{(Pz)^Tz} \label{eq:thm3l1} \\
		&\le \frac{1}{1-\beta(\mathcal{C}_k)} + \eta^2 \label{eq:thm3l2} \\
		&= \frac{4}{4-k} + \eta^2 \; . \label{eq:thm3l3}
		\end{align}
		
		The bound \eqref{eq:thm3l1} follows from all latencies being nonnegative, \eqref{eq:thm3l2} follows from \cite{correa2008geometric} and \cite{perakis2007price} (see the comment about the price of anarchy for costs with no constant term), and the \eqref{eq:thm3l3} is proved in the proof of Theorem \ref{thm:pairwise_separable}. 
		
		For part (b), we use the same notion of $\beta(\mathcal{C})$ as in the proofs for Theorems \ref{thm:correa} and \ref{thm:pairwise_separable}, as follows:	
		\begin{align*}
		\beta(c,v) &= \max_{z\in\mathbb{R}_{\ge 0}^{2n}}\frac{\langle c(v)-c(z),z \rangle}{\langle c(v),v \rangle} \\
		&= \frac{\max_{z\in\mathbb{R}_{\ge 0}^{2n}} \langle (Q+P)(v-z),z \rangle}{\langle (Q+P)v + b,v \rangle} \\
		&\le \frac{\max_{z\in\mathbb{R}_{\ge 0}^{2n}} \langle Q(v-z),z \rangle}{\langle (Q+P)v + b,v \rangle} + \frac{\max_{z\in\mathbb{R}_{\ge 0}^{2n}} \langle P(v-z),z \rangle}{\langle (Q+P)v + b,v \rangle} \\
		&\le \frac{\max_{x\in\mathbb{R}_{\ge 0}^{2n}} \langle Q(v-z),z \rangle}{\langle Qv + b,v \rangle} + \frac{\max_{z\in\mathbb{R}_{\ge 0}^{2n}} \langle P(v-z),z \rangle}{\langle Pv + b,v \rangle} \\
		&= \beta(c_1,v) + \beta(c_2,v)
		\end{align*}
		Here $c_1$ and $c_2$ represent cost functions drawn from $\mathcal{C}_k$ and $\tilde{\mathcal{C}}$, respectively, where $k$ is the maximum degree of asymmetry of the cost function $c(z) = Qz+b$ and $\tilde{\mathcal{C}}$ denotes the set of monotone cost functions.
		
		De Palma and Nesterov \cite{depalma1998optimization} show that a cost function $c(z)$ is monotone if $c'(z)$ is positive definite. Furthermore, Correa \emph{et. al.} show that a class $\mathcal{C}$ consisting of monotone cost functions has $\beta(\mathcal{C}) \le 1$. This is easily demonstrated as follows. Using \eqref{eq:monotone} with $z,v \in \mathbb{R}^{2n}_{\ge 0}$,
		\begin{align*}
		\langle c(v)-c(z), z \rangle &\le \langle c(v)-c(z), v \rangle \; ,  \\
		\end{align*}
		thus,
		\begin{align*}
		1 & \ge \frac{\langle c(v)-c(z), z \rangle}{\langle c(v)-c(z), v \rangle} \\ 
		& \ge \frac{\langle c(v)-c(z), z \rangle}{\langle c(v), v \rangle} \\
		& \ge \beta(c,v)
		\end{align*}
		Because of this,
		\begin{align*}
		\beta(\mathcal{C}) &= \sup_{c\in \mathcal{C},v\in \mathcal{X}} \beta(c,v) \\
		&\le \sup_{c\in \mathcal{C}_k,v\in \mathcal{X}} \beta(c,v) + \sup_{c\in \tilde{\mathcal{C}},v\in \mathcal{X}} \beta(c,v) \\
		&\le \frac{k}{4} + 1 \; .
		\end{align*}
		
		Here $\tilde{\mathcal{C}}$ denotes monotone cost functions. Applying Theorem \ref{thm:correa} completes the proof.
	\end{proof}
	
	\section{TIGHTNESS OF THE BOUND}\label{sct:tightness}
	
	In this section, we discuss the tightness of the bound derived in Section \ref{sct:our_bound}. In Section \ref{sct:examples} we provide two examples: Example \ref{ex:two_sided} shows that our price of anarchy is tight for $k=2$ and our bicriteria bound is tight for $k=4$ when there can be two-sided asymmetry, \emph{i.e.} $k_i$ can be greater or less than 1. In a more realistic scenario, we expect autonomous vehicles to result in the same amount or less congestion than regular cars for all roads. In light of this, we provide Example \ref{ex:two_sided} of \emph{one-sided} asymmetry, in which $k_i\ge 1$ $\forall i$. In Section \ref{sct:bound_discussion}, we discuss the tightness of the bound with respect to both of these scenarios.
	
	\subsection{Examples}\label{sct:examples}
	
	\begin{example}\label{ex:two_sided}
		Consider the traffic network in Fig. \ref{fig:twoSidedAsym}, which is parameterized by the degree of asymmetry, $k$. We wish to transport 1 unit regular traffic and 1 unit smart traffic across the network.
		
		\begin{figure}
			\begin{center}
				\begin{tikzpicture}[->, >=stealth', auto, semithick, node distance=7cm]
				\tikzstyle{every state}=[fill=white,draw=black,thick,text=black,scale=1]
				\node[state]    (0)               {$s$};
				\node[state]    (1)[right of=0]   {$t$};
				\path
				(0) edge[bend left]		node[above]{$c_1(x,y) = kx+y$}     (1)
				(0) edge[bend right]		node[below]{$c_2(x,y) = x+ky$}     (1);
				\end{tikzpicture}
			\end{center}
			\caption{Example of a road network with two-sided asymmetry.}
			\label{fig:twoSidedAsym}
		\end{figure}
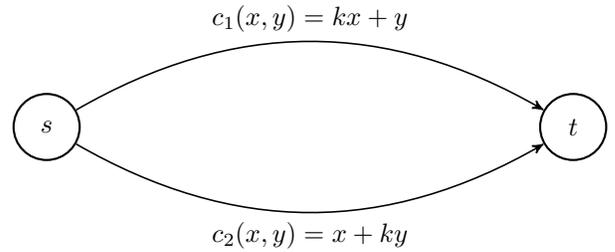		
		
		The worst-case Nash equilibrium has all regular traffic on the top link and all the smart traffic on the bottom link, for a cost of $C^{NE} = 2k$. The optimal routing has this routing reversed, for a cost of $C^{OPT}=2$. This gives us $\frac{C^{EQ}}{C^{OPT}} = k$.
		
		We find the bicriteria by finding how much traffic we could optimally route for a cost of $2k$. Consider $p$ units regular and $p$ units of smart vehicles, which would have optimal routing cost $2p^2$. Setting $2p^2=2k$, we find the bicriteria is $\sqrt{k}$.
	\end{example}
	
	\begin{example}\label{ex:one_sided}
		Consider the traffic network in Fig. \ref{fig:oneSidedAsym}, which is parameterized by $k$. Here we wish to transport $\frac{2\sqrt{k} - 1}{2k}$ units regular traffic and $\frac{1}{2}$ units smart traffic across the network.
		
		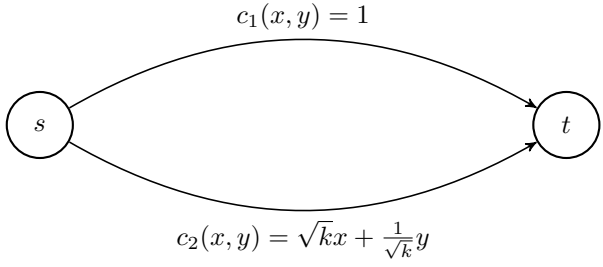
\begin{figure}
			\begin{center}
				\begin{tikzpicture}[->, >=stealth', auto, semithick, node distance=7cm]
				\tikzstyle{every state}=[fill=white,draw=black,thick,text=black,scale=1]
				\node[state]    (0)               {$s$};
				\node[state]    (1)[right of=0]   {$t$};
				\path
				(0) edge[bend left]		node[above]{$c_1(x,y) = 1$}     (1)
				(0) edge[bend right]	node[below]{$c_2(x,y) = \sqrt{k}x+\frac{1}{\sqrt{k}}y$}     (1);
				\end{tikzpicture}
			\end{center}
			\caption{Example of a road network with one-sided asymmetry.}
			\label{fig:oneSidedAsym}
		\end{figure}		
		
		At the Wardrop Equilibrium, all traffic will take the bottom route for a delay of $1$, which gives us cost
		\begin{align*}
		C^{EQ} &= \frac{2\sqrt{k} - 1}{2k} + \frac{1}{2} = \frac{k + 2 \sqrt{k} - 1}{2k}
		\end{align*}
		
		In optimal routing, we have regular traffic on top, smart traffic on the bottom. This gives us
		\begin{align*}
		C^{OPT} &= \frac{2\sqrt{k} - 1}{2k} +  \frac{1}{2\sqrt{k}} \frac{1}{2} = \frac{5\sqrt{k}-2}{4k} \; ; \\
		\frac{C^{EQ}}{C^{OPT}} &= \frac{2k + 4 \sqrt{k} - 2}{5\sqrt{k}-2}
		\end{align*}
		
		We find the bicriteria by setting the cost of routing $p$ times as much traffic optimally equal to the original cost at equilibrium. This gives us $p=\sqrt{\frac{2 (k-3) \sqrt{k}+1}{k}+8}+\frac{1}{\sqrt{k}}-2$.	
	\end{example}
	
	\subsection{Discussion}\label{sct:bound_discussion}
	
	\begin{figure}
		\includegraphics[width=1\linewidth]{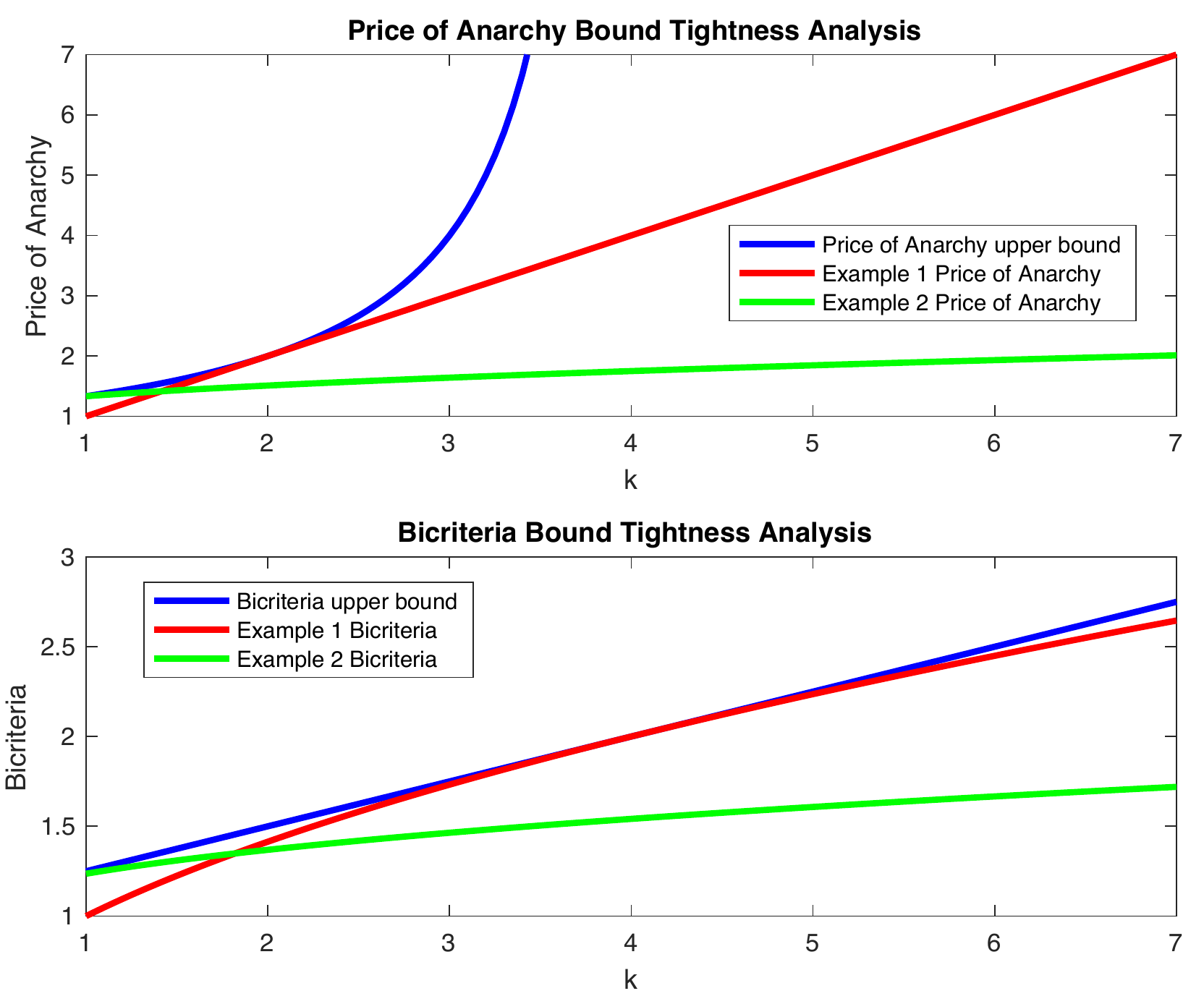}
		\caption{Tightness of the bounds for price of anarchy and bicriteria. We see that the PoA bounds is tight for $k=1$ and $k=2$, and the bicriteria bound is tight for $k=4$.}
		\label{fig:soln_seq}
	\end{figure}
	
	We begin by discussing the price of anarchy. Our bound for price of anarchy is $\frac{4}{4-k}$, and example 1 shows $\frac{C^{NE}}{C^{OPT}} = k$ and example 2 shows $\frac{C^{NE}}{C^{OPT}} = \frac{2k + 4\sqrt{k}-2}{5\sqrt{k}-2}$. The first example shows our bound is tight for $k=2$ and the second example shows our bound is tight for $k=1$. Indeed, when $k=1$ we recover the classical price of anarchy bound found in \cite{roughgarden2002bad}.
	
	For the bicriteria, our bound is $1+\frac{k}{4}$. Example 1 provides a bicriteria of $\sqrt{k}$ and example 2 has a bicriteria of $\sqrt{\frac{2 (k-3) \sqrt{k}+1}{k}+8}+\frac{1}{\sqrt{k}}-2$. Example 1 shows our bound is tight for $k=4$ and example 2 shows our bound is nearly tight for $k=1$. Furthermore, we see that this bound matches that in \cite{correa2008geometric} for $k=1$.
	
	Figure \ref{fig:soln_seq} illustrates these comparisons. In both cases, our upper bound diverges from these lower bounding examples for large $k$. Therefore, it is unknown if our bound is tight in that regime. However, realistic circumstances lead to $k \approx 2.5$, which is in our nearly-tight region for both price of anarchy and bicriteria.
	
	It is worth noting that under the construction in \cite{correa2008geometric} and in Theorem \ref{thm:correa}, there can be no bound on the price of anarchy for networks with $k \ge 4$. Observe that in Example \ref{ex:two_sided} for $k = 4$, the bicriteria is $2$. This means that $\beta(\mathcal{C}_{k=4}) \ge 1$, so the bound on the price of anarchy does not hold.
	
	\section{CONCLUSIONS}\label{sec:conclusions}
	
	In this paper, we have presented pairwise separable and nonseparable cost functions for traffic networks under mixed autonomy. We demonstrate that the price of anarchy and bicriteria is unbounded without constraints on the asymmetry in the difference in how the additon of smart and regular vehicles affects congestion. We then established bounds for the price of anarchy and bicriteria, parameterized by the degree of asymmetry of the network, for both the case of pairwise separable and nonseparable costs, under certain conditions. We analyze the tightness of the bound for the pairwise separable case and demonstrate that they are tight for certain degrees of asymmetry of the network.
	
	
	
	
	
	
	
	
	\addtolength{\textheight}{-11cm}   


	\bibliographystyle{ieeetr}
	\bibliography{biblio.bib}

\begin{thebibliography}{10}

\bibitem{dot:2015zr}
{\em Beyond Traffic 2045: Trends and Choices}.
\newblock US Department of Transportation, 2015.

\bibitem{Lioris:2015lq}
J.~Lioris, R.~Pedarsani, F.~Y. Tascikaraoglu, and P.~Varaiya, ``Platoons of
  connected vehicles can double throughput in urban roads,'' {\em arXiv
  preprint arXiv:1511.00775}, 2015.

\bibitem{Shladover:1978lo}
S.~E. Shladover, ``Longitudinal control of automated guideway transit vehicles
  within platoons,'' {\em Journal of Dynamic Systems, Measurement, and
  Control}, vol.~100, no.~4, pp.~302--310, 1978.

\bibitem{Darbha:1999dw}
S.~Darbha and K.~Rajagopal, ``Intelligent cruise control systems and traffic
  flow stability,'' {\em Transportation Research Part C: Emerging
  Technologies}, vol.~7, no.~6, pp.~329--352, 1999.

\bibitem{Yi:2006hb}
J.~Yi and R.~Horowitz, ``Macroscopic traffic flow propagation stability for
  adaptive cruise controlled vehicles,'' {\em Transportation Research Part C:
  Emerging Technologies}, vol.~14, no.~2, pp.~81--95, 2006.

\bibitem{Pueboobpaphan:2010qe}
R.~Pueboobpaphan and B.~van Arem, ``Driver and vehicle characteristics and
  platoon and traffic flow stability: Understanding the relationship for design
  and assessment of cooperative adaptive cruise control,'' {\em Transportation
  Research Record: Journal of the Transportation Research Board}, no.~2189,
  pp.~89--97, 2010.

\bibitem{Orosz:2016hb}
G.~Orosz, ``Connected cruise control: modelling, delay effects, and nonlinear
  behaviour,'' {\em Vehicle System Dynamics}, pp.~1--30, 2016.

\bibitem{Askari:2016fy}
A.~Askari, D.~A. Farias, A.~A. Kurzhanskiy, and P.~Varaiya, ``Measuring impact
  of adaptive and cooperative adaptive cruise control on throughput of
  signalized intersections,'' {\em arXiv preprint arXiv:1611.08973}, 2016.

\bibitem{Vander-Werf:2002fh}
J.~Vander~Werf, S.~Shladover, M.~Miller, and N.~Kourjanskaia, ``Effects of
  adaptive cruise control systems on highway traffic flow capacity,'' {\em
  Transportation Research Record: Journal of the Transportation Research
  Board}, no.~1800, pp.~78--84, 2002.

\bibitem{Van-Arem:2006ai}
B.~Van~Arem, C.~J. Van~Driel, and R.~Visser, ``The impact of cooperative
  adaptive cruise control on traffic-flow characteristics,'' {\em IEEE
  Transactions on Intelligent Transportation Systems}, vol.~7, no.~4,
  pp.~429--436, 2006.

\bibitem{Jiang:2007rr}
R.~Jiang, M.-B. Hu, B.~Jia, R.~Wang, and Q.-S. Wu, ``Phase transition in a
  mixture of adaptive cruise control vehicles and manual vehicles,'' {\em The
  European Physical Journal B}, vol.~58, no.~2, pp.~197--206, 2007.

\bibitem{Yuan:2009th}
Y.-M. Yuan, R.~Jiang, M.-B. Hu, Q.-S. Wu, and R.~Wang, ``Traffic flow
  characteristics in a mixed traffic system consisting of {ACC} vehicles and
  manual vehicles: A hybrid modelling approach,'' {\em Physica A: Statistical
  Mechanics and its Applications}, vol.~388, no.~12, pp.~2483--2491, 2009.

\bibitem{Kesting:2010wd}
A.~Kesting, M.~Treiber, and D.~Helbing, ``Enhanced intelligent driver model to
  access the impact of driving strategies on traffic capacity,'' {\em
  Philosophical Transactions of the Royal Society of London A: Mathematical,
  Physical and Engineering Sciences}, vol.~368, no.~1928, pp.~4585--4605, 2010.

\bibitem{Arnaout:2014ul}
G.~M. Arnaout and S.~Bowling, ``A progressive deployment strategy for
  cooperative adaptive cruise control to improve traffic dynamics,'' {\em
  International Journal of Automation and Computing}, vol.~11, no.~1,
  pp.~10--18, 2014.

\bibitem{lazar2017routing}
D.~Lazar, S.~Coogan, and R.~Pedarsani, ``Capacity modeling and routing for
  traffic networks with mixed autonomy,'' in {\em Decision and Control (CDC),
  2017 IEEE 56th Conference on, to appear}, IEEE, 2017.

\bibitem{wardrop1900some}
J.~Wardrop, ``Some theoretical aspects of road traffic research,'' in {\em Inst
  Civil Engineers Proc London/UK/}, 1900.

\bibitem{Dafermos:1969qt}
S.~C. Dafermos and F.~T. Sparrow, ``The traffic assignment problem for a
  general network,'' {\em Journal of Research of the National Bureau of
  Standards B}, vol.~73, no.~2, pp.~91--118, 1969.

\bibitem{Koutsoupias:1999fs}
E.~Koutsoupias and C.~Papadimitriou, ``Worst-case equilibria,'' in {\em Annual
  Symposium on Theoretical Aspects of Computer Science}, pp.~404--413,
  Springer, 1999.

\bibitem{papadimitriou2001algorithms}
C.~Papadimitriou, ``Algorithms, games, and the internet,'' in {\em Proceedings
  of the thirty-third annual ACM symposium on Theory of computing},
  pp.~749--753, ACM, 2001.

\bibitem{roughgarden2002bad}
T.~Roughgarden and {\'E}.~Tardos, ``How bad is selfish routing?,'' {\em Journal
  of the ACM (JACM)}, vol.~49, no.~2, pp.~236--259, 2002.

\bibitem{roughgarden2002selfishthesis}
T.~Roughgarden, ``Selfish routing,'' tech. rep., CORNELL UNIV ITHACA NY DEPT OF
  COMPUTER SCIENCE, 2002.

\bibitem{correa2011wardrop}
J.~R. Correa and N.~E. Stier-Moses, ``Wardrop equilibria,'' {\em Wiley
  encyclopedia of operations research and management science}, 2011.

\bibitem{roughgarden2005selfish}
T.~Roughgarden, {\em Selfish routing and the price of anarchy}, vol.~174.
\newblock MIT press Cambridge, 2005.

\bibitem{roughgarden2003topology}
T.~Roughgarden, ``The price of anarchy is independent of the network
  topology,'' {\em Journal of Computer and System Sciences}, vol.~67, no.~2,
  pp.~341--364, 2003.

\bibitem{chau2003price}
C.~K. Chau and K.~M. Sim, ``The price of anarchy for non-atomic congestion
  games with symmetric cost maps and elastic demands,'' {\em Operations
  Research Letters}, vol.~31, no.~5, pp.~327--334, 2003.

\bibitem{perakis2007price}
G.~Perakis, ``The “price of anarchy” under nonlinear and asymmetric
  costs,'' {\em Mathematics of Operations Research}, vol.~32, no.~3,
  pp.~614--628, 2007.

\bibitem{correa2008geometric}
J.~R. Correa, A.~S. Schulz, and N.~E. Stier-Moses, ``A geometric approach to
  the price of anarchy in nonatomic congestion games,'' {\em Games and Economic
  Behavior}, vol.~64, no.~2, pp.~457--469, 2008.

\bibitem{pigou1932economics}
A.~C. Pigou, ``The economics of welfare,'' {\em McMillan\&Co., London}, 1920.

\bibitem{depalma1998optimization}
A.~De~Palma and Y.~Nesterov, ``Optimization formulations and static equilibrium
  in congested transportation networks,'' tech. rep., 1998.

\bibitem{milanes2014cooperative}
V.~Milan{\'e}s, S.~E. Shladover, J.~Spring, C.~Nowakowski, H.~Kawazoe, and
  M.~Nakamura, ``Cooperative adaptive cruise control in real traffic
  situations,'' {\em IEEE Transactions on Intelligent Transportation Systems},
  vol.~15, no.~1, pp.~296--305, 2014.

\bibitem{lioris2017platoons}
J.~Lioris, R.~Pedarsani, F.~Y. Tascikaraoglu, and P.~Varaiya, ``Platoons of
  connected vehicles can double throughput in urban roads,'' {\em
  Transportation Research Part C: Emerging Technologies}, vol.~77,
  pp.~292--305, 2017.

\bibitem{smith1979existence}
M.~J. Smith, ``The existence, uniqueness and stability of traffic equilibria,''
  {\em Transportation Research Part B: Methodological}, vol.~13, no.~4,
  pp.~295--304, 1979.

\end{thebibliography}

\end{document}